\documentclass[11pt,letterpaper]{amsart}
\usepackage[margin=1in]{geometry}
\usepackage{amsmath}
\usepackage{amssymb}
\usepackage{amsthm}
\usepackage{amscd}
\usepackage{enumerate}
\usepackage{xcolor}
\usepackage{subcaption}

\usepackage{graphicx}
\usepackage{amsmath,amsthm,amsfonts,amscd,amssymb,comment,eucal,latexsym,mathrsfs}
\usepackage{stmaryrd}
\usepackage[all]{xy}

\usepackage{epsfig}

\usepackage[all]{xy}
\xyoption{poly}
\usepackage{fancyhdr}
\usepackage{wrapfig}
\usepackage{epsfig}

\theoremstyle{plain}
\newtheorem{thm}{Theorem}[section]
\newtheorem{ithm}{Theorem}

\newtheorem{lem}[thm]{Lemma}
\newtheorem{cor}[thm]{Corollary}

\newtheorem{que}{Question}
\theoremstyle{definition}

\theoremstyle{remark}

 \def\N{{\mathbb{N}}}  \def\P{{\mathbb{P}}}  \def\R{{\mathbb{R}}}        

    \def\cE{{\mathcal{E}}}  \def\cG{{\mathcal{G}}}            \def\cS{{\mathcal{S}}} \def\cT{{\mathcal{T}}}      

\newcommand\dist{\operatorname{d}}

\newcommand\gr{\operatorname{gr}}

\begin{document}
\title[Growth dichotomy for unimodular random rooted trees]{Growth dichotomy for unimodular random rooted trees}

\author{Mikl\'{o}s Abert}
\address{\parbox{\linewidth}{MTA Alfr\'{e}d R\'{e}nyi Institute of Mathematics
Re\'{a}ltanoda utca 13-15.
H-1053 Budapest,
Hungary
}}
\email{miklos.abert@renyi.mta.hu}
\urladdr{https://www.renyi.hu/~abert/}

\author{Mikolaj Fraczyk}
\address{\parbox{\linewidth}{Faculty of Mathematics and Computer Science, Jagiellonian University, ul. ${\L}$ojasiewicza 6, 30-348 Krak{\'o}w, Poland
}}
\email{mikolaj.fraczyk@uj.edu.pl}
\urladdr{https://sites.google.com/view/mikolaj-fraczyk/home}

\author{Ben Hayes}
\address{\parbox{\linewidth}{Department of Mathematics, University of Virginia, \\
141 Cabell Drive, Kerchof Hall,
P.O. Box 400137
Charlottesville, VA 22904}}
\email{brh5c@virginia.edu}
\urladdr{https://sites.google.com/site/benhayeshomepage/home}

\subjclass[2020]{60K35; 37A30, 37A20, 46N30, 60G10}

\thanks{M. Abert acknowledges support from the KKP 139502 project, the ERC Consolidator Grant 648017 and the Lendulet Groups and Graphs research group. B. Hayes gratefully acknowledges support from the NSF grants DMS-1600802, DMS-1827376, and DMS-2000105. M. Fraczyk acknowledges support from the Dioscuri programme initiated by the Max Planck Society, jointly managed with the National Science Centre in Poland, and mutually funded by Polish the Ministry of Education and Science and the German Federal Ministry of Education and Research.}

\begin{abstract}
  We show that the growth of a unimodular random rooted tree $(T,o)$ of degree bounded by $d$ always exists, assuming its upper growth passes the critical threshold $\sqrt{d-1}$. This complements Timar's work who showed the possible nonexistence of growth below this threshold. 
  
  The proof goes as follows. By Benjamini-Lyons-Schramm, we can realize $(T,o)$ as the cluster of the root for some invariant percolation on the $d$-regular tree. Then we show that for such a percolation, the limiting exponent with which the lazy random walk returns to the cluster of its starting point always exists. We develop a new method to get this, that we call the 2-3-method, as the usual pointwise ergodic theorems do not seem to work here. We then define and prove the Cohen-Grigorchuk co-growth formula to the invariant percolation setting. This establishes and expresses the growth of the cluster from the limiting exponent, assuming we are above the critical threshold.  
\end{abstract}

\maketitle
\section{Introduction}

A random rooted graph $(\cG,o)$ is \emph{unimodular}, if it satisfies the so-called Mass Transport Principle, saying that for any measurable paying scheme, the expected income and outpay of the root $o$ are equal. The Benjamini-Schramm limit of finite graphs is a URG, and the limit tends to behave like a vertex transitive graph in many senses. This allows one to prove new results on arbitrary graph sequences, pulling mathematical energy from the vertex transitive world (see \cite{mKesten}, \cite{LPBook} and the references therein). In general, results that hold for vertex transitive graphs tend to generalize to the unimodular random setting. 

At the Banff workshop “Graphs, groups and stochastics” in 2011,  Kaimanovich asked if the exponential growth
\[\lim_{n\to\infty}|B(o,n)|^{1/n}\]
almost surely exists for a unimodular random rooted graph $(\cG,o)$ where $B(o,n)$ denotes the ball of radius $n$ in $\cG$. This trivially holds for a fixed vertex transitive graph by submultiplicativity of the size of balls. Tim\'{a}r showed that this was false in \cite{Timar14}, giving a counterexample that was a unimodular random rooted tree of bounded degree. There was a consensus in the community that this is where the story ends.

In Tim\'{a}r's construction, the random tree is one ended a.s., the lower growth always equals $1$, and the upper growth is almost surely a small positive constant. We initially proved (see Theorem \ref{thm:HFGrowth} below) that a one-ended unimodular random rooted tree of maximal degree $d$ has upper growth at most $\sqrt{d-1}$. To our surprise, we then also found that once the upper growth of the tree surpasses this threshold, the growth will actually start to exist!

\begin{ithm}\label{thm:SGrowth}
Let $(T,o)$ be a unimodular random rooted tree with degree bound $d$ and upper growth
\[\limsup_{n\to\infty}|B(o,n)|^{1/n} > \sqrt{d-1}\]
Then the growth
\[\lim_{n\to\infty}|B(o,n)|^{1/n}\]
exists.
\end{ithm}

The proof starts by applying the theorem of Benjamini, Lyons and Schramm \cite{BLS}, saying that every unimodular random rooted tree with degree bound $d$ can be realized as the cluster (connected component) of the root of some invariant edge percolation of the $d$-regular tree $T_d$. The next step is to establish a notion of ``dimension" or ``exponent" of the cluster of the root for such invariant percolations, using the return exponent of the lazy random walk starting at the root, to the cluster of the root. 

\begin{ithm}\label{thm:Exponent}
Consider an invariant edge percolation $\cE$ of the $d$-regular tree $T_d$, rooted at $o$ and let $\cE(o)$ be the connected component of $o$. Let $(X_n)_{n\in\mathbb N}$ be the lazy random walk on $T_d$, starting at $o$. Then the limit 

\[\rho_\cE:=\lim_{n\to\infty}\P[X_{n}\in \cE(o)]^{1/n}\]

exists. 
\end{ithm}

Just by its definition, the above limit could depend on the realization of the unimodular random rooted tree  $(T,o)$ as an invariant percolation cluster, but it is easy to show that it only depends on $(T,o)$ itself. In particular, if $(T,o)$ is ergodic, then this limiting exponent is a constant. \smallskip 

\noindent {\bf Remark.} Note that we only use the lazy random walk instead of the simple one because it avoids periodicity issues. All the results of the paper can be adapted to simple random walk. \smallskip 

As the usual pointwise ergodic theorems did not seem to do the job here, in the next step we develop a new method to show the existence of the above exponent. It turns out that this method (that we call the 2-3 method) and also the existence of the exponent applies in a much greater generality than stated here, for arbitrary invariant random partitions of countable groups, and even more generally, for arbitrary subrelations of p.m.p. equivalence relations. As one would expect, the above exponent can be interpreted as the spectral radius (norm) of a suitable operator, but that takes its natural form only in the language of subrelations. The subrelation language as well as the most general results are developed in the paper \cite{AFH}. In the present paper, we do not use operator theory language and decided to make the paper self-contained by giving a probabilistic proof for the minimal version of the 2-3 method needed for percolation clusters as well. 

In the next step, we aim to express the growth of our tree $(T,o)$ in terms of the above exponent. We first use a trick that allows us to assume that the tree has no leaves. For this, we substitute the original rooted tree with its \emph{spine}, defined as the union of its bi-infinite geodesics. Note that the spine can be also obtained by consecutively removing the leaves from $(T,o)$). 

\begin{ithm}\label{thm:HFGrowth spine}
Let $(T,o)$ be a unimodular random rooted tree  of degree at most $d$. If the upper growth of $(T,o)$ is larger than $\sqrt{d-1}$ then the spine of $(T,o)$ has the same growth as $(T,o)$.\end{ithm}

In particular, in the threshold we are looking at, the growth will exist for $(T,o)$ as long as it exists for its spine. Passing to the spine makes the sizes of spheres in $(T,o)$ to be monotone increasing, that will be useful later. 

At this point we take out a well-known tool from the group theory toolbox. If we considered a fixed subgroup $H$ of the free group $F$ as a substitute of $\cE(o)$ then the above sampling exponent would trivially exist by submultiplicativity and be equal to the spectral radius of the corresponding random walk operator on $l^2(F/H)$. The growth of $H$ embedded in $F$ is called the \emph{co-growth} of $H$. In the realm of subgroups, the connection between this co-growth and the spectral radius was established by Grigorchuk and Cohen, leading to the famous co-growth formula \cite{Grig,Cohen}. 

It turns out that one can generalize this formula to our invariant percolation setting and prove that in the range where the upper growth of $(T,o)$ is $(\sqrt{d-1},d-1]$, the growth and the exponent of $(T,o)$ completely determine each other. Note that in the subgroup case of Grigorchuk and Cohen, the existence of the co-growth was automatic from submultiplicativity and the original analysis in \cite{Grig,Cohen} leading to the co-growth formula only detected the upper growth. To detect the actual growth, we have to use a large deviation principle for the distance traversed by a random walk on a regular tree \cite[Thm 19.4]{Woess}. We also have to exploit, in a nontrivial way, concavity of the rate function as well as the above explained monotonicity of the growth sequence. This part of the paper seems to be restricted to connected components of percolated trees, that is, it does not seem to generalize to subrelations, but we expect there to be further applications of this method.

As we mentioned, Timar's example for the nonexistence of growth is a one ended unimodular random rooted tree, in particular, it is hyperfinite. Recall that a unimodular random rooted tree  is hyperfinite if and only if it has one or two ends a.s. \cite[Theorem 8.9]{AldousLyons}. Our next theorem, already mentioned at the start, shows that a hyperfinite unimodular random rooted tree  has upper growth at most $\sqrt{d-1}$.

\begin{ithm}\label{thm:HFGrowth}
Let $(T,o)$ be a hyperfinite unimodular random rooted tree  of degree at most $d$. If $(T,o)$ is two ended then its growth exists and equals $1$. If $(T,o)$ is one ended then its upper growth is at most $\sqrt{d-1}$ and by Timar's construction, its growth may not exist.\end{ithm}

Note that the upper bound $\sqrt{d-1}$ for the (upper) growth of one ended unimodular random rooted trees can be achieved by the so-called \emph{canopy tree}, the Benjamini-Schramm limit of metric balls in $T_d$. \smallskip 

Theorems \ref{thm:SGrowth} and  \ref{thm:HFGrowth} still leave open the case of unimodular random rooted trees with infinitely many ends (or, equivalently, the non-hyperfinite case), but with upper growth at most $\sqrt{d-1}$. We could not decide whether these trees will also have growth or not.

\begin{que}\label{que:main}
Let $(T,o)$ be a unimodular random rooted tree with bounded degree and with infinitely many ends a.s. Does the growth
\[\lim_{n\to\infty}|B(o,n)|^{1/n}\]
exist?
\end{que}

This question is actually two questions packed together that both make sense in themselves. We discuss these and some initial results at the end of the paper. \smallskip  

The paper is structured as follows. In Section \ref{sec:Prelim} we define some of the basic notions of the paper and prove Theorem \ref{thm:HFGrowth}. In Section \ref{sec:Unim23} we establish the 2-3 method and show how it implies Theorem \ref{thm:Exponent}. In Section \ref{sec:Spine} we define the spine and prove Theorem \ref{thm:HFGrowth spine}. In Section \ref{sec:LDP} we establish the large deviation principle result needed for our main Theorem \ref{thm:SGrowth}, which is then proved in Section \ref{sec:GrowthInTrees}. Finally, in Section \ref{sec:questions} we suggest some questions and further directions. \smallskip 

\noindent \textbf{Acknowledgements.} Much of the work by the last named author was done on visits to the R\'{e}nyi institute in Budapest. He would like to thank the R\'{e}nyi Institute for its hospitality.

\section{Preliminaries}\label{sec:Prelim}

In this section we define some basic notions, like growth and unimodularity, discuss invariant percolation and prove the one-ended case of Theorem \ref{thm:HFGrowth}. This proof serves as a good practice of the Mass Transport Principle. The two-ended case will be proved in Subsection \ref{sec:Spine}.

Let $(G,o)$ be a connected rooted graph of bounded degree. We define the lower and upper growth of $(G,o)$:
\[\underline{\gr}(G)=\liminf_{n\to\infty} |B_G(o,n)|^{1/n},\qquad \overline{\gr}(G)=\limsup_{n\to\infty} |B_G(o,n)|^{1/n}.\]
It is easy to see that they do not depend on the choice of the root. If the above quantities coincide, we say that $G$ has growth and write $$\gr(G)=\underline{\gr}(G)=\overline{\gr}(G).$$
The growth of balls is completely determined by the growth of spheres. We have
\[ \gr(G)=\max\{1, \lim_{n\to\infty} |S_G(o,n)|^{1/n}\},\] and a similar identity for the upper and lower growth. 

For general rooted graphs, even trees, of bounded degree, there is no reason why the graph should have growth. In contrast, vertex transitive graphs trivially have growth. Introducing unimodular random rooted graphs can be done in various ways and here we choose the most pragmatic version, picking the Mass Transport Principle (that we use anyways later) as the definition. An in depth introduction to unimodular random rooted graphs can be found in \cite{AldousLyons}. 

Let $\mathcal G_d$ denote the space of rooted, connected graphs with degree bound $d$. A point in this space is a rooted, connected graph up to rooted isomorphism. One can define a metric on $\mathcal G_d$ as 
$$d((G_{1},o_{1}),(G_{2},o_{2}))= \inf \frac{1}{k+1}$$
over $k$ such that the $k$-ball in $G_1$ around $o_{1}$ is isomorphic to the $k$-ball in $G_2$ around $o_{2}$ (as rooted graphs). Endowed with this metric, $\mathcal G_d$ becomes a compact, totally disconnected space. 

A random rooted graph $(G,o)$ is \emph{unimodular} if it satisfies the \emph{Mass Transport Principle}. This asserts that for every measurable non-negative real valued function $F$ defined on the set of triples $(G,x,y)$, where $G$ is a bounded degree graph and $x,y$ are vertices of $G$, we have
\[ \mathbb E\left(\sum_{x\in V(\mathcal G)}F(o,x)\right)=\mathbb E\left(\sum_{x\in V(\mathcal G)}F(x,o)\right),\] provided that one of the sides is finite. A unimodular random rooted tree  is a unimodular random rooted graph that is a tree almost surely.
It is convenient to think of the above identity as the fact that for any paying scheme, the expected mass ``sent out" from the root is equal to the expected mass received by the root. 

We say that the unimodular random rooted graph $(\mathcal G,o)$ is ergodic if one can not obtain it as a nontrivial convex combination of other unimodular random rooted graphs. Since the upper and lower growth of a graph are measurable functions on $\mathcal G_d$ that do not depend on the choice of the root, they both must be constants by ergodicity, almost surely. 

Unimodular random rooted graphs admit a strong statistical type of homogeneity, and as such, they tend to behave like vertex transitive graphs. So it is reasonable to ask whether they have growth. This was a question of Kaimanovich answered negatively by Adam Timar in \cite{Timar14}. Timar constructed an example of a unimodular random rooted tree  $(\mathcal T,o)$ such that $\underline{\gr}(\mathcal T)=1$ while $\overline{\gr}(\mathcal T)=c>1$. The example in
\cite{Timar14} is a one ended tree with a hierarchy of fast and slow growing finite regions of super-exponentially growing size. The fact that the resulting graph is one ended seemed necessary in the construction and we were unable to find any example of a unimodular random rooted tree with more than one end and no growth. As we realized, for one ended trees there is a natural barrier on the upper growth. We now restate and prove Theorem \ref{thm:HFGrowth}. 

\begin{thm}\label{thm:OneEnded}
Let $(\mathcal T,o)$ be a one ended unimodular random rooted tree  of degree at most $d$. Then $\overline{\gr}(\mathcal T)\leq \sqrt{d-1}$. 
\end{thm}
\begin{proof}
Orient the edges of the tree $\mathcal T$ towards its unique end $\xi$. Call a vertex $y$ an \emph{$r$-ascendant} of $x$ if there is an oriented path of length $r$ from $x$ to $y$ and an \emph{$r$-descendant} of $x$ if $x$ is an $r$-ascendant of $y$. 

For every $r\geq 0$ consider the mass transport function $F_{r}\colon V(\mathcal T)\times V(\mathcal T)\to \mathbb R_{\geq 0}$:
\[F_r(x,y)=\begin{cases} 1 & \text{ if } y \text{ is the }r\text{-ascendant of } x,\\
0 & \text{ otherwise.}
\end{cases}\]
Under this paying scheme, the total outpay of every vertex is $1$. Let $f_r(y):=\sum_{x\in V(\mathcal T)}F_r(x,y)$ be the total income of $y$. Then $f_r(y)$ equals the number of $r$-descendants of $y$. 

By the mass transport principle, we have
\[ \mathbb E\left( f_r(o)\right)=\mathbb E\left( \sum_{y\in V(\mathcal T)}F_r(o,y)\right)=1.\]
We use $f_r$ to construct a new mass transport function $E_r$:
\[E_r(y,z)=\begin{cases}
f_{r-k}(y) & \text{ if }y\text{ is the }k\text{-ascendant of }z\\
0 & \text{otherwise,}\\\end{cases}\]
with the convention that $f_j(y)=0$ for $j<0$.
We claim that the total income $\sum_{y\in V(\mathcal T)}E_r(y,z)$ is an upper bound for $|S_\mathcal{T}(z,r)|.$ 
Indeed, every point of the $r$-sphere around $z$ is a $(r-k)$-descendant of the $k$-ascendant of $z$ for some $k=0,\ldots,r$. 

We now estimate the expected total outpay of the root:
\begin{align*}\mathbb E\left(\sum_{z\in V(\mathcal T)} E_r(o,z)\right)&=\mathbb E\left(\sum_{k=0}^r f_k(o)f_{r-k}(o)\right)\\
&\ll \mathbb E\left(\sum_{k=1}^{\lfloor r/2\rfloor}(d-1)^kf_{r-k}(o)+\sum_{k=\lfloor r/2\rfloor+1}^r f_{k}(o)(d-1)^{r-k}\right)\\
&=\sum_{k=1}^{\lfloor r/2\rfloor}(d-1)^k\mathbb E(f_{r-k}(o))+\sum_{k=\lfloor r/2\rfloor+1}^r \mathbb E(f_{k}(o))(d-1)^{r-k}\ll r(d-1)^{r/2}.
\end{align*}

By the mass transport principle we have
\[\mathbb E(|S_{\mathcal T}(o,r)|)\leq \mathbb E\left(\sum_{z\in V(\mathcal T)} E_r(o,z)\right)\ll  r(d-1)^{r/2}.\]
Let $\alpha>\sqrt{d-1}$. By Markov's inequality we get
\[ \mathbb P\left( |S_{\mathcal T}(o,r)|\geq \alpha^r\right)\ll r\left(\frac{\sqrt{d-1}}{\alpha}\right)^r.\]
By the Borel-Cantelli lemma we deduce that the event that $|S_{\mathcal T}(o,r)|\geq \alpha^r$ happens only finitely many times almost surely. Therefore $\overline{\gr}(\mathcal T)\leq \alpha.$ We prove the theorem by letting $\alpha$ go to $\sqrt{d-1}.$
\end{proof}

\section{The $2$-$3$--method for percolated trees}\label{sec:Unim23}

In this section we prove Theorem \ref{thm:Exponent} with an extra ergodicity statement. To do this, we introduce and prove the 2-3 method, using Mass Transport. 

Let $\cT_d$ be the $d$-regular tree with a root $o$ and let $\cE$ be an invariant edge percolation. For any vertex $v\in \cT_d$ let $\cE(v)$ be the cluster of $v$ 
(i.e. the connected component of $v$ in the graph whose edge set is $\cE$).
Finally, let $X_n$ be the lazy random walk on $\cT_d$ starting at $o$. Our goal in this section is to prove that the limit 
$$\rho_{\cE}:=\mathbb{P}[X_{2n}\in \cE(o)]^{1/2n}$$
always exists. To motivate this statement one may observe that when we identify $\cT_d$ with the Cayley graph of the free group $F_{d/2}$ and the connected components of $\cE$ happen to be orbits of a subgroup $H\subset F_{d/2}$ then the limit above is nothing else that the spectral radius of the lazy random walk on the coset graph $F_{d/2}/H$. For a general percolation it is not clear why such limit should exist, as the connected components of $\cE$ aren't as homogeneous as cosets of a subgroups. Still, the invariance of the percolation turns out to be enough. 

The following theorem is essentially Theorem \ref{thm:Exponent} from the Introduction, adding a further ergodicity statement. 

\begin{thm}\label{ThmA}
The limit $$\lim_{n\to\infty}\mathbb{P}[X_{2n}\in \cE(o)]^{\frac{1}{2n}}$$ exists.
Moreover if $\cE$ has indistinguishable clusters, then the limit is almost surely equal to $$\rho^\cE:=\lim_{n\to\infty}(\mathbb E[1_{X_{2n}\in \cE(o)}])^{1/2n}$$
\end{thm}
For any $x,y\in \cT_d$ write $p_{k}(x,y)$ for the probability that the lazy random walk starting at $x$ ends in $y$ in $k$ steps. Let $p^{\cE}_k(x):=\sum_{y\in\cE(x)}p_k(x,y).$
We define two functions $f^\cE_1, f^\cE_2\colon \cT_d \times 2\N \to \R_{\geq 0}$.
\begin{align*}f^\cE_1(x,2l):=&\sum_{y\in \cE(x)} p_{2l}(x,y)p^\cE_{2l}(y)^{-1}\\
f^\cE_2(x,2l):=& p_{2l}^\cE(x)\sum_{y\in \cE(x)}p_{2l}(x,y)\left(\sum_{z\in \cE(x)} p_{2l}(y,z)p^{\cE}_{2l}(z)\right)^{-1}.\\
\end{align*}

\begin{lem}\label{lem-MTP1} For any $x\in \cT_d$
\begin{align*}
\mathbb E[f^\cE_1(x,2l)]=&1,\\
\mathbb E[f^\cE_2(x,2l)]=&1.\\
\end{align*}
\end{lem}
\begin{proof}
Let $F_1, F_2\colon \cT_d\times \cT_d\to \R_{\geq 0}$ be defined by 
\[F_1(x,y):=p_{2l}(x,y)p_{2l}^\cE(x)^{-1}, \qquad F_2(x,y):=p_{2l}(x,y)p_{2l}^\cE(y)\left(\sum_{z\in \cE(x)}p_{2l}(x,z)p_{2l}^\cE(z)\right)^{-1}.\]

We easily check that $\sum_{y\in\cE(x)} F_i(x,y)=1$ for $i=1,2$ and every $x\in \cT_d$. By the mass transport principle we get 
\[ 1=\mathbb E\left[\sum_{y\in \cE(x)} F_1(y,x)\right]=\mathbb E\left[\sum_{y\in \cE(x)} F_2(y,x)\right].\]
The lemma follows, since 
\[f^\cE_1(x,2l)=\sum_{y\in \cE(x)} F_1(y,x),\qquad f^\cE_2(x,2l)=\sum_{y\in \cE(x)} F_2(y,x).\]
\end{proof}
\begin{lem}\label{lem23} We have 
\begin{align*}
p_{2l}^\cE(x)^2\leq& f_1^\cE(x,2l)p_{4l}^\cE(x),\\
p_{2l}^\cE(x)^3\leq& f_2^\cE(x,2l)p_{6l}^\cE(x).
\end{align*}
\end{lem}
\begin{proof}
By Cauchy-Schwartz 
\begin{align*} p_{2l}^\cE(x)^2\leq& \left(\sum_{y\in\cE(x)} p_{2l}(x,y)p_{2l}^\cE(y)^{-1}\right)\left(\sum_{y\in\cE(x)} p_{2l}(x,y)p_{2l}^\cE(y)\right)\leq f^\cE_1(x,2l) p_{4l}^\cS(x).\\
p_{2l}^\cE(x)^2\leq& \left(\sum_{y\in\cE(x)} p_{2l}(x,y)\left(\sum_{z\in \cE(x)} p_{2l}(y,z)p^{\cE}_{2l}(z)\right)^{-1}\right)\left(\sum_{y\in\cE(x)} p_{2l}(x,y)\left(\sum_{z\in \cE(x)} p_{2l}(y,z)p^{\cE}_{2l}(z)\right)\right)\\
\leq&  p_{2l}^\cE(x)^{-1}f^\cE_2(x,2l) p_{6l}^\cE(x).
\end{align*}
\end{proof}
\begin{proof}[Proof of Theorem \ref{ThmA}]
Let $l,p,q\in\N$. By Lemma \ref{lem23} and a simple induction, for every $x\in X$ we have 
\begin{equation}\label{eq-Inductive23}p_{2l}^\cE(x)^{2^p3^q}\leq C_{p,q}(x,2l) p_{2^{p+1}3^ql}^\cE(x),\end{equation} where \[C_{p,q}(x,2l):=\prod_{i=0}^{p-1}f_1^\cE(x,2^{i+1}l)^{2^{p-1-i}3^q}\prod_{j=0}^{q-1}f_2^\cE(x,2^{p+1}3^jl)^{3^{q-1-j}}.\]
By Lemma \ref{lem-MTP1} the sums $ \sum_{l=1}^\infty l^{-2} f_i^{\cE}(x,2l), i=1,2$ converge a.e. In particular, for almost every $x$ there exists an $l_0=l_0(x)$ such that $f_i^\cE(x,2l)\leq l^2$ for $l\geq l_0$. Hence, for $l\geq l_0$ we will have 
\begin{align*}C_{p,q}(x,2l)\leq& \prod_{i=0}^{p-1} (2^{i+1}l)^{2^{p-i}3^q}\prod_{j=0}^{q-1}(2^{p+1}3^jl)^{2\cdot 3^{q-1-j}},\\
2^{-p}3^{-q}\log C_{p,q}(x,2l)\leq& \sum_{i=0}^{p-1} \frac{\log(2^{i}l)}{2^{i}}+\sum_{j=0}^{q-1} \frac{\log(2^{p}3^jl)}{2^p3^j}\\
\leq& 50+2\log l.
\end{align*}
Let $l_1\geq l_0$. Put $\underline\rho_s^\cE:=\liminf_{l\to\infty}p^{\cE}_{2l}(x)^{\frac{1}{2l}}$ and $ \overline\rho_s^\cE:=\limsup_{l\to\infty}p^{\cE}_{2l}(x)^{\frac{1}{2l}}$. Since the set $2^{p}3^ql_1$ becomes dense on the logarithmic scale as $p,q\to\infty$, we have \begin{align*}
\underline{\rho}_x^\cE=\liminf_{l\to\infty}p^{\cE}_{2l}(x)^{\frac{1}{2l}}=&\liminf_{p,q\to\infty} p^{\cE}_{2^{p+1}3^ql_1}(x)^{1/2^{p+1}3^{q}l_1^{1}}\\
\geq& \liminf_{p,q\to\infty} C_{p,q}(x,2^{p+1}3^ql_1)^{-1/2^{p+1}3^ql_1} p_{2l_1}^\cE(x)^{1/2l_1}\\
\geq& l_1^{1/l_1} p_{2l_1}^\cE(x)^{1/2l_1}.\end{align*} By taking the limit along a sequence $l_1\to \infty$ such that $p_{2l_1}^\cE(x)^{{1}/{2l_1}}\to \overline\rho_x^\cE$ we obtain the inequality $\underline{\rho}_x^\cE\geq \overline\rho_x^\cE$. This proves that $\lim_{l\to\infty}p_{2l}^\cE(x)^{1/2l}$ exists almost surely. 

We now prove the second part of the theorem. Assume $\cE$ has indistinguishable clusters and define $\rho^{\cE}(x):=\lim_{l\to\infty}p_{2l}^\cE(x)^{1/2l}$. This function is constant of the clusters of $\cE$, so by indistinguishability is almost surely equal to some constant $\rho_0$. 
By (\ref{eq-Inductive23}) and the inequality between the arithmetic and the geometric mean we get \begin{align*}\frac{p_{2l}^\cE(x)}{p_{2^{p+1}l}^\cE(x)^{1/2^p}}\leq C_{p,0}(x,2l)^{1/2^p}=\left(\prod_{i=0}^{p-1} f_1^\cE(x,2^{i+1}l)^{2^{p-1-i}}\right)^{1/2^p}
\leq \frac{1}{2^p}\left(1+\sum_{i=0}^{p-1} 2^{p-1-i}f_1^\cE(x,2^{i+1}l)\right).
\end{align*}
We take expectation and use Lemma \ref{lem-MTP1} to get 
\[\mathbb E\left[\frac{p_{2l}^\cE(x)}{p_{2^{p+1}l}^\cE(x)^{1/2^p}}\right]\leq 1.\]
Upon taking the limit as $p\to \infty$ we find that $\mathbb E\left[ p_{2l}^\cE(x)\right]\leq \rho_0^{2l}$. Therefore, $\limsup_{l\to\infty} \mathbb E\left[ p_{2l}^\cE(x)\right]\leq \rho_0$. By Fatou lemma we have $\liminf_{l\to\infty} \mathbb E\left[ p_{2l}^\cE(x)\right]\geq \rho_0$, so $\lim_{l\to\infty} \mathbb E\left[ p_{2l}^\cE(x)\right]= \rho_0$ and the theorem is proved.
\end{proof}

\section{The spine of a unimodular random rooted tree}\label{sec:Spine}

In this section we introduce the spine and prove Theorem \ref{thm:HFGrowth spine}. 

Let $T$ be a tree. Call an edge $e$ \emph{weak} if one side of $T\setminus e$ is finite. The spine of $T$, denoted ${\rm spine}(T)$, is the unique infinite connected component of what remains of $T$ after throwing out all the weak edges.
The spine of a tree $T$ is naturally a sub-tree of $T$. If the tree $T$ is not one ended, then $T\setminus {\rm spine}(T)$ is a union of finite trees. We will call these trees \emph{decorations}.

Let $(\mathcal T,o)$ be a unimodular random rooted tree  and let $\nu$ be the corresponding probability measure on the moduli space $\mathcal M_d$ of rooted graphs of degree bounded by $d$. Let $C=\{(G,o): G \text{ is a tree and }o\text{ is in the spine}\}$. This is a measurable set and $\nu(C)>0$ unless $(\mathcal T,o)$ is one ended a.s.. We define the spine of $(\mathcal T,o)$ as $({\rm spine}(\mathcal T_c),o_c)$ where $(\mathcal T_c,o_c)$ is chosen according to $\nu|_C$. It is a unimodular random rooted graph.

\begin{thm}\label{thm:spineGrowth}
Let $(\cT,o)$ be a unimodular random rooted tree  of degree bounded by $d$. Let $(\cT',o')$ be the spine of $(\cT,o)$. If $\cT'$ is regular, then $\overline{gr}(\cT)=\overline{gr}(\cT')$. In general, we have $\overline{gr}(\cT)\leq \max\{ (d-1)^{1/2}, \overline{gr}(\cT')\}$. In particular, if $\overline{gr}(\cT)\geq (d-1)^{1/2}$ then $\overline{gr}(\cT)=\overline{gr}(\cT').$
\end{thm}

The starting point of our discussion is the following lemma.
\begin{lem}\label{decorations} Let $(\cT,o)$ be a unimodular random rooted tree  and let $(\cT',o')$ be its spine, which we assume to be non-empty almost surely. Then, the expected size of the decoration adjacent to the root $o'$ is finite.
\end{lem}
\begin{proof}
 For every vertex $v\in {\rm spine}(T)$ let $w(v)-1$ be the total number of vertices of the connected components of $T\setminus {\rm spine}(T)$ adjacent to $v$.

We define a function on pairs of vertices: $F(x,y)=1$ if $x\in {\rm spine}(T)$ and $y=x$ or $y$ is in a connected component of $T\setminus {\rm spine}(T)$ adjacent to $x$ and $F(x,y)=0$ otherwise. It is easy to see that for any $x$ in the spine we have $\sum_{y\sim x} F(x,y)=w(x)$ and for any vertex $y$ we have $\sum_{x\sim y} F(x,y)\leq 1$. Let $(\cT,o)$ be a unimodular random rooted tree  and let $(\cT',o')$ be the spine. By the mass transport principle we have
$${\mathbb P}(o\in \cT'){\mathbb E}(w(o'))={\mathbb E}\sum_{x\sim o'} F(o',x)={\mathbb E}\sum_{x\sim o'} F(x,o')\leq 1.$$
Since we assume that $\cT'$ is non empty, the probability ${\mathbb P}(o\in \cT')$ is non-zero. The lemma is proved.
\end{proof}
We are ready to prove Theorem \ref{thm:spineGrowth}.
\begin{proof}
Let us start with the easier case, when $\mathcal T'$ is a $d'$-regular tree, $d'\leq d$. Let $w\colon \mathcal T'\to \mathbb N$ be the weight of decorations, defined as at the beginning of the proof of Lemma \ref{decorations}. Let $\alpha>1$ and let $F\colon \mathcal T'\times \mathcal T'\to \mathbb R_{\geq 0}$ be given as
\[ F(x,y)= (\alpha(d'-1))^{-\dist(x,y)} w(x).\]
By the mass transport principle, we have
\[ {\mathbb E}\left(\sum_{x\in \cT'} F(x,o')\right)={\mathbb E}\left(\sum_{x\in \cT'} F(o',x)\right)\leq \frac{d'}{d'-1}(1-\alpha^{-1})^{-1} {\mathbb E}(w(o')).\]
In particular the leftmost sum is finite a.s.. We unfold it to get
\begin{align*}\left(\sum_{x\in \mathcal T'} F(x,o')\right)&=\sum_{r=0}^\infty (\alpha(d'-1))^{-r}\sum_{x\in S_{\mathcal T'}(r)}w(x)\\ \gg &\sum_{r=0}^\infty (\alpha(d'-1))^{-r}\sum_{x\in B_{\cT'}(r)}w(x) \geq \sum_{r=0}^\infty (\alpha(d'-1))^{-r} B_{\mathcal T}(o',r).\end{align*}
We have used the fact that $|B_{\mathcal T}(o',r)|\leq \sum_{x\in B_{\mathcal T'}(o',r)} w(x)$. This crude estimate will be sufficient for regular $\mathcal T'$.  We deduce that for every vertex $o'$ in the spine of $\cT$ we have $\limsup_{r\to\infty} |B_{\cT}(o',r)|^{1/r}\leq \alpha (d-1)$ almost surely. The origin $o$ is at a finite distance from the spine so we get that $\overline{gr}(\cT)\leq  \alpha (d'-1)$. We get the regular case of theorem by letting $\alpha\to 1$.

We move to the general case. Let $(\mathcal T,o)$ be a unimodular random rooted tree  of degree at most $d$ with the spine $(\mathcal T',o')$. If $(\mathcal T',o')$ is empty almost surely then the graph $\mathcal T$ is one-ended so the theorem follows from Lemma \ref{thm:OneEnded}. From now on we assume that $(\cT',o')$ is non-empty almost surely. Let $C:=\max\{ \overline{gr}(\cT'), (d-1)^{1/2}\}$. We need to show that $\overline{gr}(\mathcal T)\leq C$. Let $\alpha>1$. We start with a simple observation. For any vertex $x\in \mathcal T'$ we have
\begin{align}\label{eq-Ballspine1} |B_{\mathcal T}(x,r)|\leq& \sum_{y\in B_{\mathcal T'}(x,r)} \min\{w(y), (d-1)^{r-\dist(x,y)}\}\\ \leq& \sum_{y\in B_{\mathcal T'}(x,r)} (w(y) (d-1)^{r-\dist(x,y)})^{1/2}.\nonumber \end{align}
Let $R$ be a positive integer. Let $H_R\colon \cT'\times \cT'\to \mathbb R_{\geq 0}$ be the mass transport function defined as
\[ H_R(x,y)=w(x)|S_{\mathcal T'}(x,\dist(x,y))|^{-1} 1_{\dist(x,y)\leq R}.\]
We compute the mass sent out from the origin $o'$:
\[\sum_{y\in \mathcal T'} H_R(o',y)=(R+1) w(o').\]
Hence, by the mass transport principle
\begin{align} {\mathbb E}\left (\sum_{y\in B_{\mathcal T'}(o',R)} w(y) |S_{\mathcal T'}(y,\dist(o',y))|^{-1} \right)=&{\mathbb E}\left (\sum_{y\in \mathcal T'} H_R(y,o')\right) \\=& (R+1)\mathbb{E}(w(o')).\nonumber\end{align} We recall that $\mathbb{E}(w(o'))<\infty$ by Lemma \ref{decorations}.
By the Borel-Cantelli lemma, the inequality
\begin{equation}\label{ineq2}
\sum_{y\in B_{\mathcal T'}(o',R)} w(y) |S_{\mathcal T'}(y,\dist(o',y))|^{-1}=\sum_{r=0}^{R} \sum_{y\in S_{\mathcal T'}(o',r)}w(y) |S_{\cT'}(y,r)|^{-1} \leq \alpha^R,
\end{equation} holds for all but finitely many $R\in \mathbb Z_{\geq 0}$ almost surely.

We would like to obtain a similar estimate for the expression $$\sum_{y\in B_{\mathcal T'}(o',R)} (d-1)^{R-\dist(o',y)}|S_{\mathcal T'}(y,\dist(o',y))|= \sum_{r=0}^{R} (d-1)^{R-r}\sum_{y\in S_{\mathcal T'}(o',r)} |S_{\mathcal T'}(y,r)|.$$
We have \[\sum_{y\in S_{\mathcal T'}(o',r)} |S_{\mathcal T'}(y,r)|\leq \sum_{k=0}^r (d-1)^{r-k}|S_{\mathcal T'}(o',2k)|,\] because every vertex $x$ with $\dist(x,y)=2k\leq 2r$, is counted with multiplicity at most $|S_{\mathcal T'}(z,r-k)|$, where $z$ is the midpoint of the geodesic segment connecting $y$ to $x$. Hence,
\begin{align*}
\sum_{r=0}^{R} (d-1)^{R-r}\sum_{y\in S_{\mathcal T'}(o',r)} |S_{\mathcal T'}(y,r)| \leq& \sum_{r=0}^{R} (d-1)^{R-r}\sum_{k=0}^r (d-1)^{r-k}|S_{\mathcal T'}(o',2k)|\\ =& \sum_{k=0}^{R} (R-k)(d-1)^{R-k}|S_{\cT'}(o',2k)|\\ \leq& R\sum_{k=0}^{R}(d-1)^{R-k}|S_{\cT'}(o',2k)|.
\end{align*}
We have $\limsup_{k\to \infty}|S_{\mathcal T'}(o',2k)|^{1/(2k)}\leq \overline{\gr}(\mathcal T')$ almost surely, so
\begin{align}
\sum_{y\in B_{\mathcal T'}(o',R)} (d-1)^{R-\dist(o',y)}|S_{\mathcal T'}(y,\dist(o',y))|\leq & R\sum_{k=0}^{R}(d-1)^{R-k}|S_{\cT'}(o',2k)|\label{ineq4}\\ \leq& \alpha^{R} C^{2R}.\nonumber
\end{align} holds for all but finitely many $R\in \mathbb Z_{\geq 0}$ almost surely.

Let $R\in \mathbb Z_{\geq 0}$ be such that both (\ref{ineq2}),(\ref{ineq4}) hold. By the Cauchy-Schwarz inequality, we get
\[\sum_{y\in B_{\mathcal T'}(o',R)} (w(y) (d-1)^{R-\dist(o',y)})^{1/2} \leq \alpha^{R}C^{R}.\]
Both  (\ref{ineq2}),(\ref{ineq4}) hold for almost all $R\in \mathbb Z_{\geq 0}$ almost surely so by combining (\ref{eq-Ballspine1}) with the last inequality we get that
\begin{align*}\limsup_{R\to \infty} |B_{\mathcal T}(o',R)|^{1/R}\leq& \limsup_{R\to\infty} \left( \sum_{y\in B_{\mathcal T'}(o',R)} (w(y) (d-1)^{R-\dist(o',y)})^{1/2}\right)^{1/R}\\ \leq& \alpha C \end{align*} holds almost surely. To prove the theorem, we let $\alpha\to 1$.
\end{proof}

\section{A large deviation principle for the lazy random walk on a regular tree.}\label{sec:LDP}

In this section we establish a large deviation principle for the lazy random walk on a regular tree. The starting result is due to Lalley, which we quote from the writeup in Woess's book \cite[19.4]{Woess}, but we need to add some additional math to be able to apply it. 

Let $\mathcal T_d$ be the $d$-regular tree. Choose a root $o\in \mathcal T_d$ and let $X_n$ be the lazy random walk starting at $o$. The probability of passing to each neighbor is $1/2d$ and we do not move with probability $1/2$. Let $x$ be a vertex of $\mathcal T_d$. We have the following estimate.
\begin{thm}\cite[19.4]{Woess}
There is an analytic function $\varphi\colon[0,1]\to\mathbb R$ such that
\[ \mathbb P[X_n=x]\sim B(\dist(o,x)/n)(1+\frac{d-2}{2}\dist(x,o))n^{-3/2}e^{n\varphi(\dist(x,o)/n)},\]
where $B\colon [0,1]\to \mathbb R$ is an analytic function, positive on $(0,1)$. Here the sign $\sim$ means that the ratio is asymptotically contained between two positive constants. These estimate holds uniformly in $[\varepsilon,1-\varepsilon]$ for any $\varepsilon>0.$
\end{thm}
We have $\varphi(0)=\log\left(\frac{1}{2}+\frac{\sqrt{d-1}}{d}\right), \varphi(1)=-\log(2d)$.
\begin{figure}
    \centering
    \includegraphics{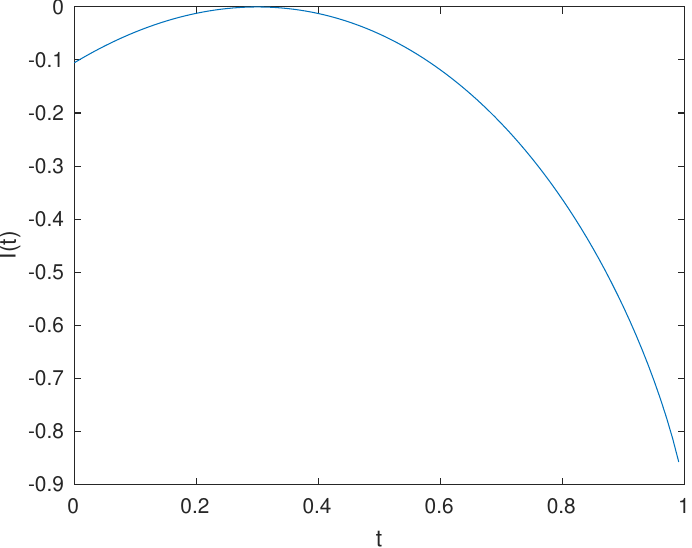}
    \caption{Rate function for $d=5$}
    \label{fig:RateF}
\end{figure}
Put $I(t)=\varphi(t)+t\log(d-1)$ for $t\in[0,1]$. Since the distribution of $X_n$ is spherically symmetric, we have $\mathbb P(d(X_n,o)=\dist(x,o))=d(d-1)^{\dist(x,o)}\mathbb P(X_n=x).$ The result of Woess implies the following large deviations estimate.
\begin{cor}
 Let $a<b\in[0,1]$. Then
\[\lim_{n\to\infty}\frac{1}{n}\log \mathbb P(an\leq \dist(o,X_n)\leq bn)=\max_{t\in [a,b]}I(t).\]
\end{cor}
We will refer to $I(t)$ as the \emph{rate function} (Figure \ref{fig:RateF}).  It is analytic and non-positive on $[0,1]$.

\begin{lem}\label{lem:RateDerivative}
The function $I$ is strictly concave, and the right derivative of $I$ at $0$ satisfies $I'(0)=\frac{1}{2}\log(d-1).$
\end{lem}
\begin{proof}
For $d\in \N$, set
\[r=\left(\frac{1}{2}+\frac{\sqrt{d-1}}{d}\right)^{-1},\,\, s=\left(\frac{1}{2}-\frac{\sqrt{d-1}}{d}\right)^{-1}\]
and define
\[F\colon [0,r]\to \R\]
by
\[F(x)=\frac{d}{(d-1)x}\left(\left(1-\frac{1}{2}x\right)-\sqrt{(1-x/r)(1-x/s)}\right).\]
The function $\phi\colon [0,1]\to \R$ is given by \cite[19.3]{Woess}
\[\phi(t)=\min\{t\log(F(x))-\log(x):0\leq x\leq r\}.\]
Set $\psi\colon [0,1]\to \R$,$x\colon [0,1]\to R$ by
\[\psi(t)=\sqrt{d^{2}t^{2}+4(d-1)(1-t^{2})},\]
\[x(t)=\frac{2d}{(d-2)^{2}}(d-\psi(t)).\]
Then by \cite[19.3]{Woess}, we have that the minimum defining $\phi$ is achieved at $x(t)$. Note that for $t\in (0,1)$, we have that $x(t)\in (0,r)$, moreover $x(0)=r,x(1)=0$. Since $x(t)\in (0,r)$ for all $t\in (0,1)$ we have that
\[0=\frac{d}{dx}(t\log F(x)-\log(x))\big|_{x=x(t)}=t\frac{F'(x(t))}{F(x(t))}-\frac{1}{x(t)} \textnormal{ for all $t\in (0,1)$}.\]
Thus
\begin{equation}\label{log diff simplification}
\phi'(t)=\log(F(x(t)))+t\frac{F'(x(t))}{F(x(t))}x'(t)-\frac{x'(t)}{x(t)}=\log F(x(t)) \textnormal{ for all $t\in (0,1)$}.
\end{equation}
Thus the right derivative of $\phi$ at $0$ is given by:
\begin{align*}\lim_{t\to 0}\phi'(t)=&\log(F(x(0)))=\log(F(r))=\log\left(\frac{d}{(d-1)r}\left(1-\frac{1}{2}r\right)\right)\\ =&\log\left(\frac{d}{d-1}\left(\frac{1}{r}-\frac{1}{2}\right)\right)=\log\left(\frac{1}{\sqrt{d-1}}\right).\end{align*}
Hence, the right derivative of $\phi$ at $0$ is $-\frac{1}{2}\log(d-1)$. We proceed to compute the second derivative of $\phi$. Using (\ref{log diff simplification}), we have:
\[\phi''(t)=\frac{F'(x(t))}{F(x(t))}x'(t)=\frac{x'(t)}{tx(t)}=-\frac{2d}{(d-2)^{2}}\frac{\psi'(t)}{tx(t)} \textnormal{ for all $t\in (0,1)$.}\]
Since $x(t)>0$ for all $t\in (0,1)$, and $\psi'(t)>0$ for all $t\in (0,1)$ this shows that $\phi$ is strictly concave.

\end{proof}

\section{Growth of unimodular random rooted trees}\label{sec:GrowthInTrees}

In this section we prove our main Theorem \ref{thm:SGrowth}, using all the results before. Since we already have working notation for the upper and lower growth, we restate Theorem \ref{thm:SGrowth} using those. 

\begin{thm}\label{thm:URTGrowth} Let $(\mathcal T,o)$ be a unimodular random rooted tree  of degree at most $d$ with $\overline{\gr}(\mathcal T)>\sqrt{d-1}$. Then $\underline{\gr}(\mathcal T)=\overline{\gr}(\mathcal T)$.
\end{thm}

\begin{proof}
Let $(\mathcal T,o)$ be a unimodular random rooted tree  of degree at most $d$ such that $\overline{\gr}(\mathcal T)>\sqrt{d-1}.$ Let $(\mathcal T',o')$ be the spine of $(\mathcal T,o)$. By Theorem \ref{thm:spineGrowth} we have $\overline{\gr}(\mathcal T)=\overline{\gr}(\mathcal T')$ and the inequality $\underline{\gr}(\mathcal T')\leq \underline{\gr}(\mathcal T)$ is clear. Therefore, it is enough to prove existence of the growth for the spine. We can assume from now on that $(\mathcal T,o)$ has no leaves. We will use the following result of Benjamini, Lyons and Schramm:

\begin{thm}[{\cite[4.2]{BLS}}]\label{lem:Subtree}
There exists an invariant edge percolation $\cE$ on the $d$-regular tree $\mathcal T_d$ such that the law of $\mathcal T$ is given by the connected component of the root.
\end{thm}
Not that if $\mathcal T$ is an ergodic unimodular random rooted graph then the percolation $\cE$ has indistinguishable clusters, by definition. Theorem \ref{thm:URTGrowth} will be deduced from the existence of the existence of the exponent $\rho_\cE,$ measuring the decay of the return probability to clusters of $\cE$. We will need several preliminary lemmas.
\begin{lem}\label{lem:LipschCone}
Let $(\alpha_n)_{n\in\mathbb N}$ be a sequence of real numbers such that $|\alpha_{n+1}-\alpha_n|\leq C$ for some constant $C$. Define the set \[\mathcal C:=\bigcap_{m\to\infty}\overline{\bigcup_{n\geq m}\{(r/n,\alpha_r/n)|r\in\mathbb N\}}.\] Then $\mathcal C$ is the closed cone given by
\[\mathcal C=\{(x,y)| x\in [0,\infty), x\liminf_{n\to\infty}\frac{\alpha_n}{n}\leq y\leq x\limsup_{n\to\infty}\frac{\alpha_n}{n}\}.\]
In particular, the limit $\lim_{n\to\infty}\frac{\alpha_{n}}{n}$ exists if and only if $\mathcal C$ is a half line.
\end{lem}
\begin{proof}
It is clear that $\mathcal C$ is contained in the cone \[\{(x,y)| x\in [0,\infty), x\liminf_{n\to\infty}\frac{\alpha_n}{n}\leq y\leq x\limsup_{n\to\infty}\frac{\alpha_n}{n}\}.\] We sketch the proof of the reverse inclusion. Let $\alpha_-:=\liminf_{n\to\infty}\frac{\alpha_n}{n}, \alpha_+:=\limsup_{n\to\infty}\frac{\alpha_n}{n}$. Let $x\in [0,+\infty)$ and
$y\in [\alpha_-x,\alpha_+x].$ For the sake of contradiction, suppose that $(x,y)\not\in \mathcal C.$ Since $\mathcal C$ is closed, there is an $\varepsilon>0$ such that $B_{\mathbb R^2}((x,y),\varepsilon)\cap \mathcal C=\emptyset$. By definition of $\mathcal C$, this means that there is an $m_0\in\mathbb N$ such that $(r,\alpha_r)\not\in \bigcup_{n\geq m_0}n B_{\mathbb R^2}((x,y),\varepsilon)$ for all $r\in\mathbb N$. Write $U:=\bigcup_{n\geq m_0}n B_{\mathbb R^2}((x,y),\varepsilon).$ As it can be seen on Figure \ref{fig:Cone}, for large $r$ the set $U$ becomes too thick for the sequence $(r,\alpha_r)$ to cross from one side to the other. This is a contradiction, because the sequence must approach the lines $y=\alpha_- x$ and $y=\alpha_+x$ infinitely often.
\begin{figure}
    \centering
    \includegraphics[trim=0 200 0 200,clip,scale=0.5]{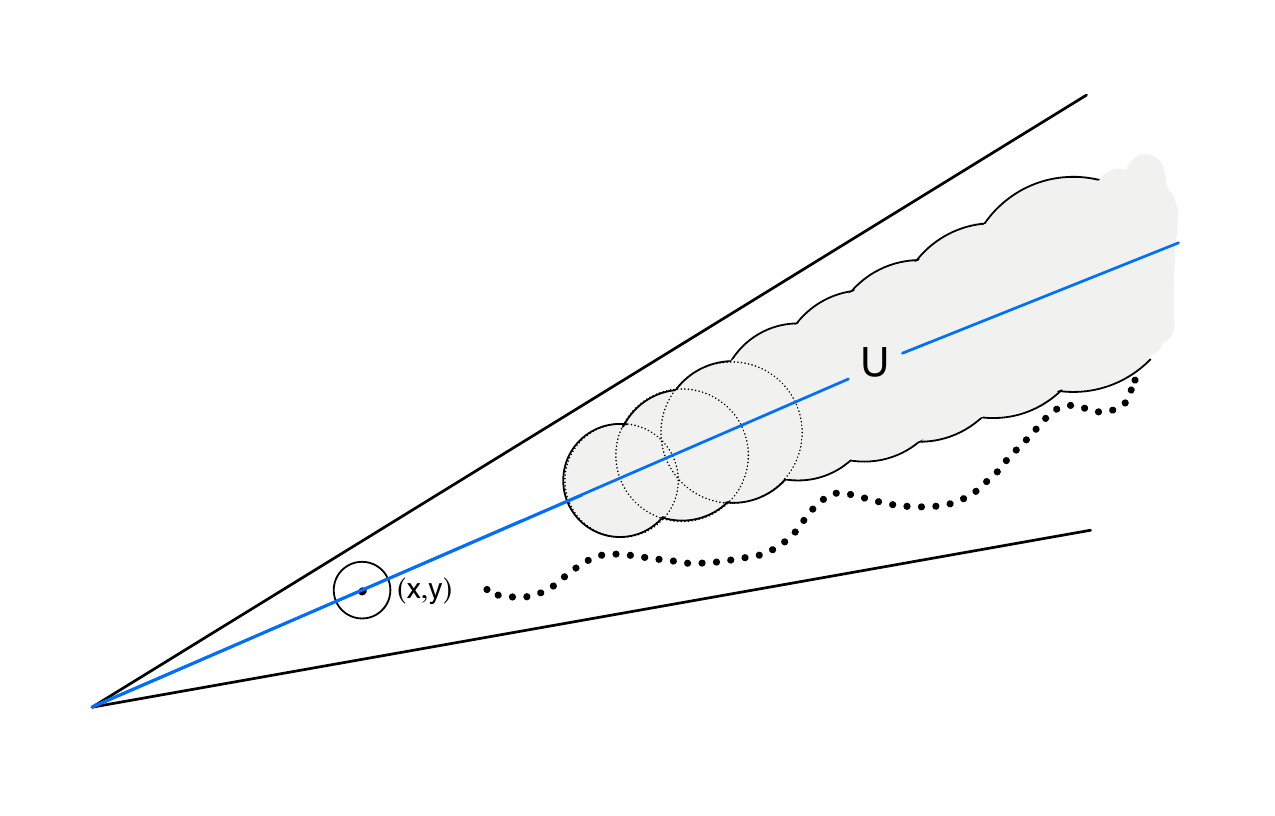}
    \caption{Sequence $(r,\alpha_r)$ and the set $U$.}
    \label{fig:Cone}
\end{figure}
\end{proof}

\begin{lem}\label{lem:ConvexConvergence}Let $(\alpha_n)_{n\in\mathbb N}$ be a sequence of real numbers and let $F\colon [0,1]\to\mathbb R$ be a strictly concave continuous function, right differentiable at $0$. Suppose that
\begin{enumerate}
    \item $\liminf_{n\to\infty}\frac{\alpha_n}{n}< F'(0),$ \label{item:derivative inequality for cone}
    \item $|\alpha_{n+1}-\alpha_n|\leq C$ for some constant $C$,
    \item the limit $\ell:=\lim_{n\to\infty} \max_{r=0,\ldots, n}\left(F\left(\frac{r}{n}\right)-\frac{\alpha_r}{n}\right)$ exists.
\end{enumerate}
Then, the limit $\lim_{n\to\infty} \frac{\alpha_n}{n}$ exists.
\end{lem}
\begin{proof}
Let \[\mathcal C=\bigcap_{m\to\infty}\overline{\bigcup_{n\geq m}\{(r/n,\alpha_r/n)|r\in\mathbb N\}}.\] Let $\alpha_-:=\liminf_{r\to\infty}\frac{\alpha_r}{r}, \alpha_+:=\limsup_{r\to\infty}\frac{\alpha_r}{r}.$ By Lemma \ref{lem:LipschCone}, $\mathcal{C}$ is the closed cone bounded by the half lines $y=\alpha_-x, x\geq 0$ and  $y=\alpha_+x,x\geq 0$.
Let $D:=\{(x,y)| x\in[0,1], y\leq F(x)-\ell\}$ and put $D_-:=\{(x,y)| x\in[0,1], y< F(x)-\ell\}$. Both are convex and $D$ is closed.
Let $t\in[0,1]$. The third condition of the lemma implies that $\mathcal C\cap D_-=\emptyset.$ It also implies that there exists a sequence  $(n_m)_{m\in\mathbb N}$, $r_m\in\{0,\ldots, n_m\}$ such that $\lim_{m\to\infty}\left(F\left(\frac{r_m}{n_m}\right)-\frac{\alpha_{r_m}}{n_m}\right)=\ell$. Passing to a sub-sequence we can assume that $\lim_{m\to\infty}(\frac{r_m}{n_m}, \frac{\alpha_{r_m}}{n_m})=(x_0,y_0).$ Condition (\ref{item:derivative inequality for cone}) and the fact that $\mathcal{C}\cap D_{-}=\varnothing,$ guarantee that $(x_0,y_0)\ne (0,0).$ By construction, $(x_0,y_0)\in \mathcal C\cap D$ so we deduce that $\mathcal C\cap D\neq \emptyset.$
\begin{figure}
\begin{subfigure}{0.5\textwidth}
    \centering
    \includegraphics[trim=0 200 0 200,clip,width=.9\linewidth]{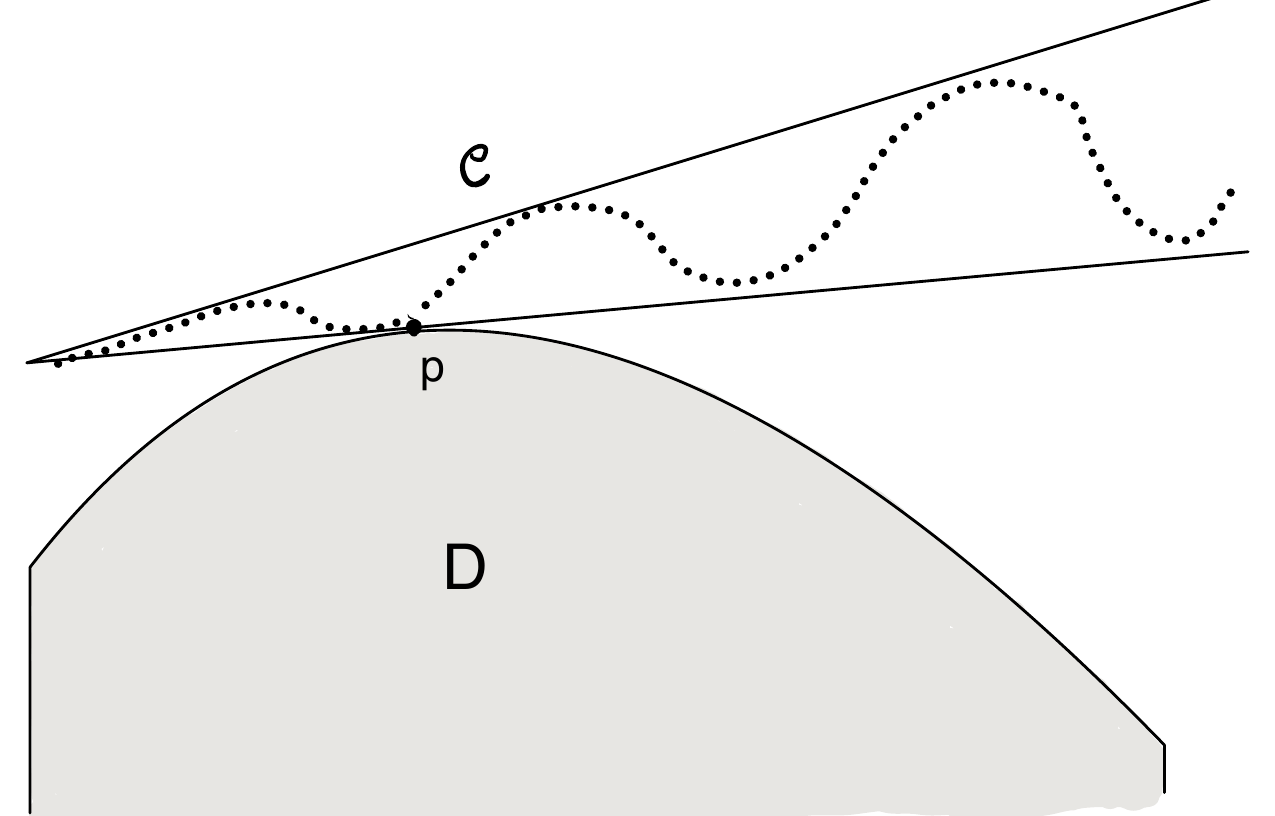}
    \caption{Sets $D$ and $\mathcal C$}
    \label{fig:DandC}
\end{subfigure}%
\begin{subfigure}{0.5\textwidth}
    \centering
    \includegraphics[trim=0 180 0 180,clip,width=.9\linewidth]{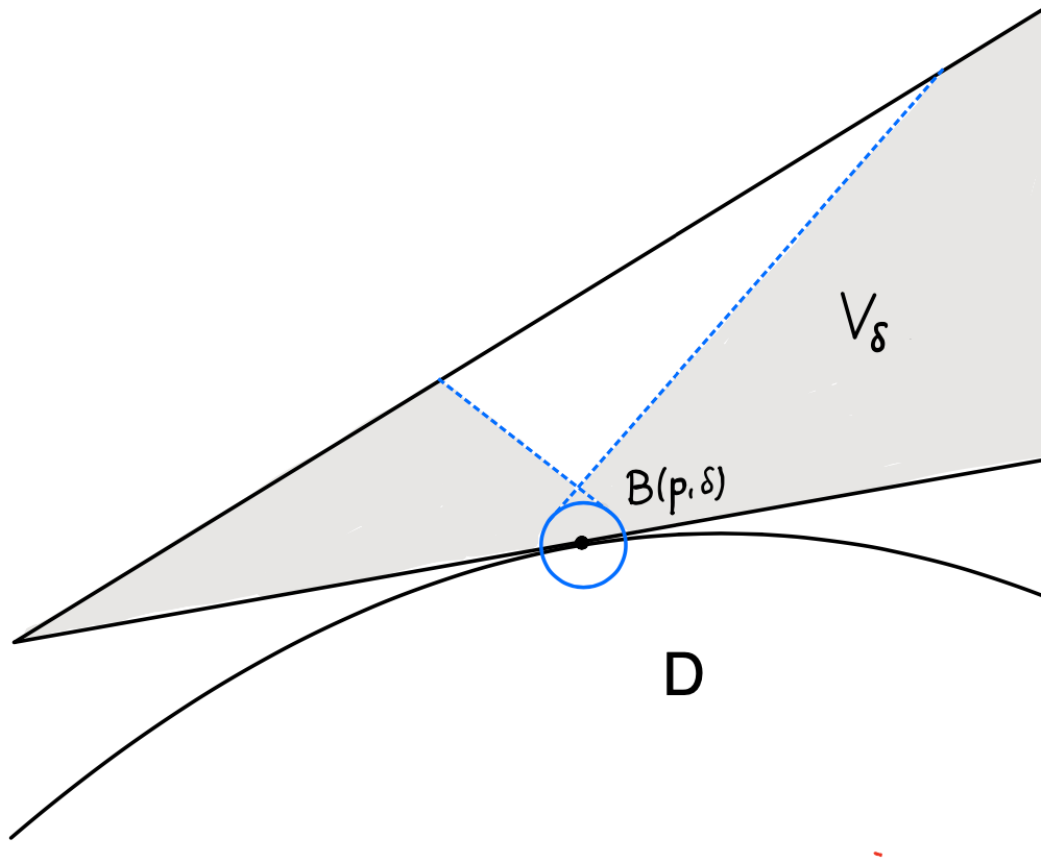}
    \caption{Set $V_\delta$}
    \label{fig:vdelta}
\end{subfigure}
\label{fig:cones}
\caption{ }
\end{figure}
By the strict concavity of $F(t)$, the intersection $\mathcal C\cap D$ is a single point $p=(x_0,y_0)$ lying on the line $y=\alpha_- x$ (see Figure \ref{fig:DandC}).

Let $\delta,\varepsilon>0$. Define $W_\varepsilon=\{(x,y)\in\mathcal C| y-F(x)+\ell\leq \varepsilon\}.$ By the strict concavity of $F(t)$ we can choose $\varepsilon$ small enough so that $W_\varepsilon\subset B_{\mathbb R^2}(p,\delta).$ Note that for every point $(x,y)\in \mathcal C\setminus W_\varepsilon$ we have $F(x)-y\leq \ell-\varepsilon$. Hence, by the condition (3), there exists an $n_0\in\mathbb N$ such that for all $n>n_0$ there is an $r_n\in\mathbb N$ such that $(\frac{r_n}{n},\frac{\alpha_{r_n}}{n})\in W_\varepsilon\subset B_{\mathbb R^2}(p,\delta)$. Let $$V_\delta:=\bigcup_{(x_1,y_1)\in B_{\mathbb R^2}(p,\delta)}\{(x,y)\in\mathcal C| |y-y_1|\leq C|x-x_1|\}$$ (see Figure \ref{fig:vdelta}). By the condition (2), all the points $(\frac{r}{n},\frac{\alpha_r}{n})$, $n\geq n_0$ are contained in $V_\delta$. Let $$U_\delta:=\bigcap_{n\geq n_0}nV_\delta,$$
(see Figure \ref{fig:vdeltaint}). From the preceding discussion we know that $(r,\alpha_r)\in U_\delta$ for every $r\in\mathbb N$.
\begin{figure}
    \centering
    \includegraphics[width=0.6\linewidth]{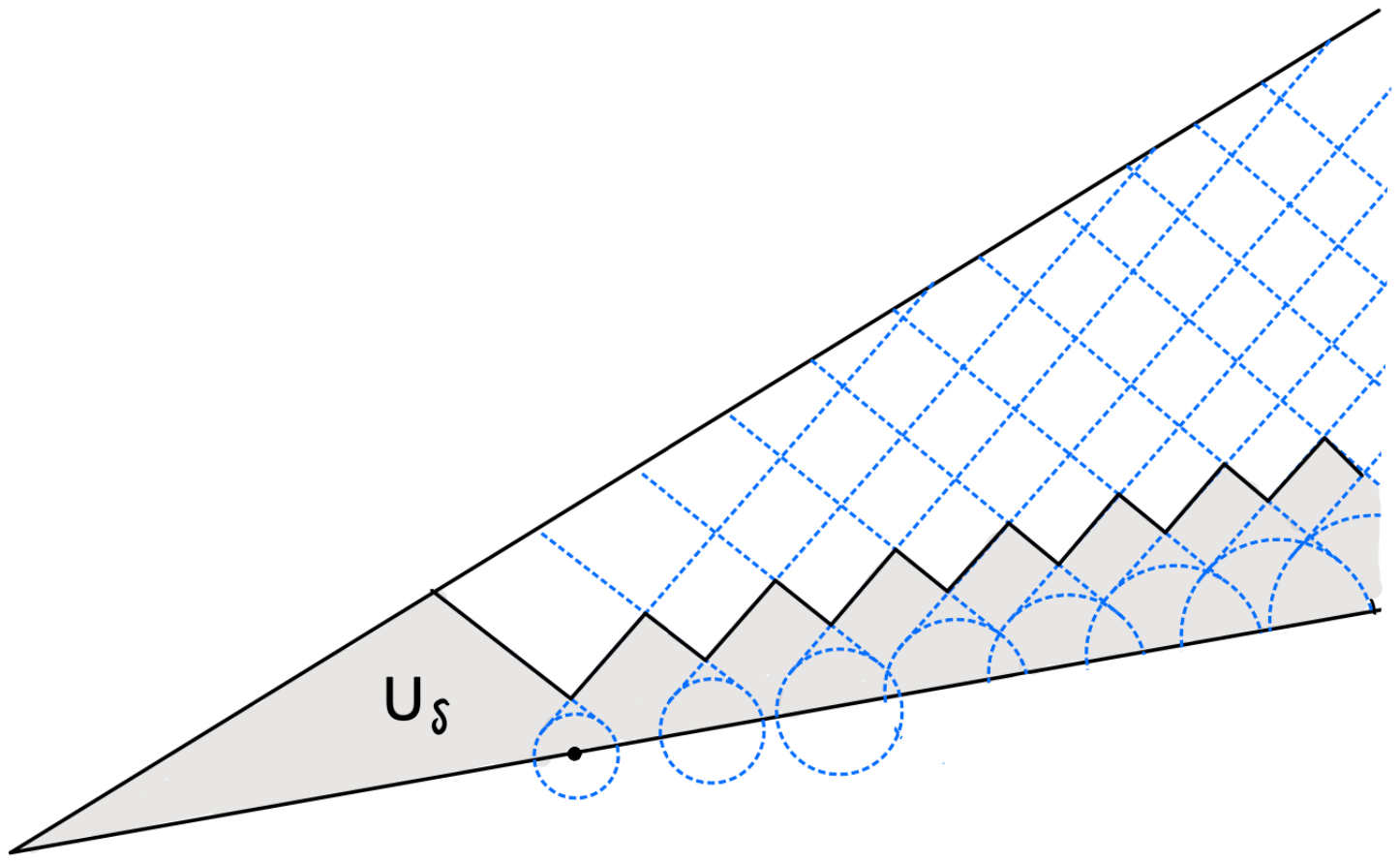}
    \caption{Set $U_\delta$}
    \label{fig:vdeltaint}
\end{figure}
Simple computation shows that this forces an inequality $\limsup_{r\to\infty}\frac{\alpha_r}{r}\leq \alpha_-+O(C\delta).$
Taking $\varepsilon,\delta\to 0$ we get $\limsup_{r\to\infty}\frac{\alpha_r}{r}=\liminf_{r\to\infty}\frac{\alpha_r}{r}.$
\end{proof}

Let $X_n$ be the lazy random walk on $\cT_d$. Using notation from Section \ref{sec:Unim23} and Theorem \ref{lem:Subtree} we have
\[p_{n}^\cE(o)=\mathbb P(X_n\in \cT).\]
Let $a_n:=\frac{|S_{\mathcal T}(o,n)|}{d(d-1)^{n-1}}$ for $n\geq 1$ and $a_0:=1$. Using the spherical symmetry of $X_n$ we can rewrite the last identity as
\[p_{n}^\cE(o)=\sum_{r=0}^n \mathbb P(\dist(X_n,o)=r)a_r.\]
By Theorem \ref{ThmA}
\[ \lim_{n\to\infty}\left(\sum_{r=0}^n \mathbb P(\dist(X_n,o)=r)a_r\right)^{1/n}=\rho^\cE.\]
The number of terms in the sum is sub-exponential in $n$, so we get
\begin{align} \log\rho^\cE=&\lim_{n\to\infty}\max_{r=0,\ldots,n} \left(\frac{1}{n}\log \mathbb P(\dist(X_n,o)=r)+\frac{1}{n}\log a_r\right)\\
=&\lim_{n\to\infty}\max_{r=0,\ldots,n} \left(I\left(\frac{r}{n}\right)+\frac{\log a_r}{n}\right).\label{eqn:MaxIdentity2}\end{align}

To prove that the growth $\gr(\cT)$ exists, we need to show that $\lim_{n\to\infty}\frac{1}{n}\log a_n$ exists. We note that the sequence $a_n$ satisfies the inequalities:
\begin{equation}\label{eqn:Lipsch} a_{n+1}\leq a_n \text{ for }n\geq 1 \text{ and } a_{n+1}\geq \frac{a_{n}}{d-1}.\end{equation}
The second inequality is in fact the only way we use the fact that $\mathcal T$ has no leaves.

Let $\alpha_n:=-\log a_n$. We would like to use Lemma \ref{lem:ConvexConvergence} to show that $\lim_{n\to\infty}\frac{-\log a_n}{n}$ exists.
By Lemma \ref{lem:RateDerivative}, $I'(0)=\frac{1}{2}\log(d-1).$ The assumption $\overline{\gr}(\mathcal T)>\sqrt{d-1}$ yields
$$\liminf_{n\to\infty}\frac{\alpha_n}{n}< \frac{1}{2}\log(d-1),$$ so the condition (1) of Lemma \ref{lem:ConvexConvergence} is satisfied.
The estimate (\ref{eqn:Lipsch}) and the identity (\ref{eqn:MaxIdentity2}) imply that the sequence $(\alpha_n)_{n\in\mathbb N}$ satisfies also the conditions (2),(3) of Lemma \ref{lem:ConvexConvergence}. Upon applying the lemma we find that $\lim_{n\to\infty}\frac{\alpha_n}{n}$ exists. The existence of the growth $\gr(\mathcal T)$ follows. \end{proof}
As a corollary, we obtain an analogue of the celebrated Cohen-Grigorchuk co-growth formula \cite{Cohen,Grig} for unimodular random rooted trees.
\begin{cor}\label{cor:cogrowth}
Let $\cT,\cE$ be as in Theorem \ref{lem:Subtree}. Let $\gamma:=\gr(\cT)$. Then \[\rho^{\cE}=\begin{cases}\frac{\sqrt{d-1}}{d}\left(\frac{e^{\gamma}}{\sqrt{d-1}}+\frac{\sqrt{d-1}}{e^{\gamma}}\right) & \text{ if } \gamma\geq \frac{1}{2}\log(d-1)\\ \frac{2\sqrt{d-1}}{d} & \text{ otherwise.}\end{cases}\].
\end{cor}
\begin{proof}
We can read the exponent $\rho^\cE$ from the formula (\ref{eqn:MaxIdentity2}). It will depend only on the asymptotic growth of the sequence $\log a_r$. To get the formula one has now to repeat the computation proving \cite[Theorem 3]{Cohen} where $\gamma_i$'s in Cohen's paper are replaced by our $|S_{\mathcal T}(o,i)|=a_id(d-1)^{i-1}$, $2t$ is our degree $d$, Cohen's $a_{i,n}$ is our $d^n\mathbb P(d(X_n,o)=i)$ and the quantity $\|a\|$ being defined by \cite[(3.3)]{Cohen} is then $d\rho^\cE$. 
\end{proof}

\section{Questions and further directions}\label{sec:questions}

As we stated in the Introduction, Question \ref{que:main} is in fact two separate questions packed together. The first is the same question, but we add the assumption that the tree equals its own spine. 

\begin{que}\label{que:spine}
Let $(T,o)$ be a unimodular random rooted tree with bounded degree, with no leafs and with infinitely many ends a.s. Is it true that the growth of $(T,o)$ exists?
\end{que}

The second asks whether passing to the spine can change having a growth or change the value of the growth. Currently we only prove this assuming the lower bound on the upper growth. 

\begin{que}
Let $(T,o)$ be a unimodular random rooted tree with bounded degree and with infinitely many ends a.s. Assume that the growth of the spine of $(T,o)$ exists. Does the growth of $(T,o)$ exist and is it equal to the growth of the spine? 
\end{que}

In other terms, if we take such a unimodular random rooted tree without leaves, can we decorate it in a umimodular way with finite trees hanging down of finite expected size, that changes either having growth, or the value of the growth?

These two questions, if answered right, will formally prove Question \ref{que:main} and we feel that they are of fundamentally different nature so it makes sense to separate then. 

It is natural to ask whether in our main theorem on the existence of growth, one can drop the assumption that $(T,o)$ is a tree. In particular, is there a similar result for arbitrary unimodular random rooted graphs? We expect that the answer is negative but were unable to produce counterexamples. 

\bibliographystyle{abbrv}
\bibliography{co-spectral}

\begin{thebibliography}{1}

\bibitem{AFH}
M.~Abert, M.~Fraczyk, and B.~Hayes.
\newblock Co-spectral radius for countable equivalence relations, 2023.

\bibitem{mKesten}
M.~Ab{\'e}rt, Y.~Glasner, and B.~Vir{\'a}g.
\newblock The measurable kesten theorem.
\newblock {\em The Annals of Probability}, 44(3):1601--1646, 2016.

\bibitem{AldousLyons}
D.~Aldous and R.~Lyons.
\newblock Processes on unimodular random networks.
\newblock {\em Electronic Journal of Probability}, 12:1454--1508, 2007.

\bibitem{BLS}
I.~Benjamini, R.~Lyons, and O.~Schramm.
\newblock Unimodular random trees.
\newblock {\em Ergodic Theory and Dynamical Systems}, 35(2):359--373, 2015.

\bibitem{Cohen}
J.~M. Cohen.
\newblock Cogrowth and amenability of discrete groups.
\newblock {\em Journal of Functional Analysis}, 48(3):301--309, 1982.

\bibitem{Grig}
R.~I. Grigorchuk.
\newblock Symmetric random walks on discrete groups.
\newblock {\em Uspekhi Matematicheskikh Nauk}, 32(6):217--218, 1977.

\bibitem{LPBook}
R.~Lyons and Y.~Peres.
\newblock {\em Probability on trees and networks}, volume~42 of {\em Cambridge
  Series in Statistical and Probabilistic Mathematics}.
\newblock Cambridge University Press, New York, 2016.

\bibitem{Timar14}
A.~Tim\'{a}r.
\newblock A stationary random graph of no growth rate.
\newblock {\em Ann. Inst. Henri Poincar\'{e} Probab. Stat.}, 50(4):1161--1164,
  2014.

\bibitem{Woess}
W.~Woess.
\newblock {\em Random walks on infinite graphs and groups}, volume 138 of {\em
  Cambridge Tracts in Mathematics}.
\newblock Cambridge University Press, Cambridge, 2000.

\end{thebibliography}

\end{document}